\documentclass[11pt]{amsart}

\usepackage{epigamath}

\usepackage[english]{babel}

\numberwithin{equation}{section}

\usepackage{amssymb}
\usepackage{mathtools}

\usepackage{enumitem}
\setlist[enumerate,1]{label={\rm(\arabic*)}, ref={\rm\arabic*}} 

\theoremstyle{plain}
\newtheorem{thm}{Theorem}[section]
\newtheorem{theorem}[thm]{Theorem}
\newtheorem{lemma}[thm]{Lemma}
\newtheorem{corollary}[thm]{Corollary}
\newtheorem{proposition}[thm]{Proposition}

\theoremstyle{definition}
\newtheorem{definition}[thm]{Definition}

\theoremstyle{remark}
\newtheorem{remark}[thm]{Remark}
\newtheorem{example}[thm]{Example}

\DeclareMathOperator{\Aut}{Aut}
\DeclareMathOperator{\Gal}{Gal}
\DeclareMathOperator{\Spec}{Spec}
\DeclareMathOperator{\id}{id}
\DeclareMathOperator{\RF}{RF}

\DeclareMathOperator{\RS}{RS}
\DeclareMathOperator{\GL}{GL}
\DeclareMathOperator{\PGL}{PGL}
\DeclareMathOperator{\PSL}{PSL}
\DeclareMathOperator{\PBD}{PBD}
\DeclareMathOperator{\diag}{diag}

\newcommand{\triv}{\mathrm{triv}}

\newcommand{\C}{{\mathbb C}}
\newcommand{\F}{{\mathbb F}}
\newcommand{\R}{{\mathbb R}}
\newcommand{\Z}{{\mathbb Z}}

\newcommand{\supth}[1]{\ensuremath{#1^{\mathrm{th}}}}

\EpigaVolumeYear{9}{2025} \EpigaArticleNr{17} \ReceivedOn{January 31, 2024}
\InFinalFormOn{July 30, 2024}
\AcceptedOn{February 9, 2025}

\title{On the number of real forms of a complex variety}
\titlemark{Real forms}

\author{Gerard van der Geer}
\address{Korteweg-de Vries Instituut, Universiteit van Amsterdam, Postbus 94248,
1090 GE  Amsterdam, The Netherlands}
\email{g.b.m.vandergeer@uva.nl}

\author{Xun Yu}
\address{Center for Applied Mathematics, Tianjin University, 92 Weijin Road, Nankai District, Tianjin 300072, P. R. China.}
\email{xunyu@tju.edu.cn}

\authormark{G.~v.d.~Geer and X.~Yu}

\AbstractInEnglish{We give a bound on the number of weighted real forms of a complex variety with finite automorphism group, where the weight is the inverse of the number of automorphisms of the real form.  We give another bound involving the (isomorphism class of a) Sylow $2$-subgroup, and as an application we give bounds on real forms of plane curves.}


\acknowledgement{The first author thanks YMSC of Tsinghua University for excellent working conditions during the autumn of 2023.  The second author is supported by the National Natural Science Foundation of China (grant nos.\ 12071337, 11831013 and 11921001).}

\begin{document}

%%%%%%%%%%%%%%%%%%%%%%%%%%%%%%%
% Title page
%%%%%%%%%%%%%%%%%%%%%%%%%%%%%%%

%\removeabove{}
%\removebetween{}
%\removebelow{}

\maketitle

\begin{prelims}

\DisplayAbstractInEnglish

\bigskip

\DisplayKeyWords

\medskip

\DisplayMSCclass

\end{prelims}

%%%%%%%%%%%%%%%%%%%%%
% Table of Contents
%%%%%%%%%%%%%%%%%%%%%

\newpage

\setcounter{tocdepth}{1}

\tableofcontents

%%%%%%%%%%%%%%%%%%%%%
% Content begins here
%%%%%%%%%%%%%%%%%%%%%

\section{Introduction}
If $X$ is a complex variety, then a real form of $X$ is a variety $Y$ defined over ${\R}$ such that $X \cong Y\times_{\Spec \R}\Spec\C$ as complex varieties.  Recently there has been a lot of activity dealing with complex varieties with infinitely many real forms. In this note we deal with quasi-projective varieties with finite automorphism group and derive a formula for the number of real forms. More precisely, for a variety over a finite field with a finite automorphism group, there is a mass formula that says that the weighted sum of its twists equals $1$.  Here we derive an analogue of this for the real forms of a complex algebraic variety.  We also give another bound involving the (isomorphism class of a) Sylow $2$-subgroup of the automorphism group.

As an application we give a bound on the number of real forms of plane curves.  Natanzon \cite{Na78} proved that a smooth projective curve $X$ of genus $g\ge 2$ has at most $2(\sqrt{g}+1)$ non-isomorphic real forms with at least one real point, and this bound is attained for infinitely many odd genera $g$. Gromadzki and Izquierdo \cite{GI98} proved that if $g$ is even, then $X$ has at most four non-isomorphic real forms with at least one real point. Our result shows that the number of non-isomorphic real forms of smooth plane curves is uniformly bounded independently of their genera and independently of having a real point or not. In particular, a smooth projective curve with at least nine non-isomorphic real forms cannot be a plane curve.

%%%%
\section{A bound on the number of weighted real forms}
Let $G:=\Gal(\C/\R)=\{\id_\C, c\}$, where $c$ is complex conjugation $z\mapsto \bar{z}$.  Let $X$ be a complex quasi-projective variety and $\pi\colon X\rightarrow \Spec \C$ the structure morphism.  A {\it real form} of $X$ is a scheme $Y$ over $\Spec \R$ such that
$$
X \cong Y \times_{\Spec \R} \Spec \C
$$
over $\Spec \C$. Two real forms are equivalent if they are isomorphic over $\Spec \R$.
The set of equivalence classes of real forms of $X$ is denoted by $\RF(X)$. 

A {\it real structure} of $X$ is an anti-holomorphic involution 
$$\sigma \colon X \longrightarrow X;$$
that is, $\sigma$ is an automorphism over $\Spec \R$ with $\sigma^2 = \id_X$ and 
$\pi \circ \sigma = c \circ \pi$. 
Two real structures $\sigma$ and~$\sigma'$ on~$X$ are said to be equivalent 
if $\sigma' = h^{-1} \circ \sigma \circ h$ for some $h \in \Aut(X/\C)$. 
We denote by  $\RS(X)$ the set of real structures 
and by $\overline{\RS}(X)$ the set of equivalence classes of real structures of $X$. 
There is a $1$-$1$ correspondence between $\overline{\RS}(X)$ and $\RF(X)$ 
given by associating to $\sigma$ the quotient variety $X/\langle \sigma\rangle$ obtained by Galois descent
(see \cite[Section~III.1.3]{Se02}).
An automorphism $f\in \Aut(X/\C)$ descends to an automorphism of $X/\langle \sigma\rangle$ over $\Spec \R$ if and only if $f\circ \sigma =\sigma\circ f$.  Thus, as abstract groups, we may identify $\Aut(X/\langle \sigma\rangle/\R)$ with the centralizer $C_{\Aut(X/\C)}(\sigma):=\{f\in \Aut(X/\C)\mid f\circ \sigma =\sigma\circ f\}.$

From now on, we fix a real form $Y_0$ and an identification $X=Y_0\times_{\Spec\R} \Spec\C$.  Clearly,
$$
\sigma_0:=\id_{Y_0}\times c\colon X\longrightarrow X
$$ 
is a real structure of $X$. Then $G$ acts on $\Aut(X/\C)$ by 
$$
f \longmapsto c\cdot f:=\sigma_0\circ f\circ\sigma_0
$$
for any $f\in \Aut(X/\C)$.  A $1$-{\it cocycle} of $G$ in $\Aut(X/\C)$ is a map $s\mapsto a_s$ from $G$ to $\Aut(X/\C)$ such that
$$
a_{st}=a_s\circ (s\cdot a_t) \quad \text{\rm for all $s, t\in G$}\, .
$$ 
Since $G$ is a cyclic group of order $2$ generated by $c$, we may naturally identify the set $Z^1(G,\Aut(X/\C))$ of such $1$-cocycles with a subset of $\Aut(X/\C)$:
$$
Z^1(G,\Aut(X/\C))=\{f\in \Aut(X/\C)\mid f\circ (c\cdot f)=\id_X\}.
$$ 
Two $1$-cocycles $f$ and $f^\prime$ are {\it cohomologous} if there exists an $h\in \Aut(X/\C)$ such that
$$
f^\prime=h^{-1}\circ f\circ (c\cdot h).
$$ 
This is an equivalence relation in $Z^1(G, \Aut(X/\C))$, and the quotient set is the {\it first cohomology set} ${\rm H}^1(G, \Aut(X/\C))$; see \cite{Se02}.  The map $f \mapsto f\circ \sigma_0$ defines a bijection $Z^1(G,\Aut(X/\C))\leftrightarrow \RS(X)$ and induces a bijection ${\rm H}^1(G, \Aut(X/\C))\leftrightarrow \overline{\RS}(X)$.  Thus we have bijections
$$
{\rm H}^1(G, \Aut(X/\C))\longleftrightarrow \overline{\RS}(X)
\longleftrightarrow \RF(X) \, .
$$
  
A mass formula for curves defined over finite fields was given in \cite[Proposition 5.1]{GV92}.  The following theorem gives an analogue for real forms of an algebraic variety.

\begin{theorem}\label{thm:formula}
Let $X/\C$ be a quasi-projective algebraic variety with real form $Y_0$.  Suppose $\Aut(X/\C)$ is a finite group. Then we have
$$
\sum_{Y\in \RF(X)}\frac{1}{\#\;\Aut(Y/\R)}=\frac{\#\; Z^1(G, \Aut(X/\C))}{\#\; \Aut(X/\C)} \leq 1\, .
$$
Moreover, if $\Aut(X/{\C})$ is not abelian, the inequality is strict.
\end{theorem}

\begin{proof}
The proof is similar to that of \cite[Proposition 5.1]{GV92}. The group $\Aut(X/\C)$ has a (right) action on $Z^1(G, \Aut(X/\C))$ given by
$$
Z^1(G, \Aut(X/\C))\times \Aut(X/\C) \longrightarrow Z^1(G, \Aut(X/\C)),
\quad (f,h)\longmapsto h^{-1}\circ f\circ (c\cdot h).
$$
The orbits correspond to the cohomology classes in ${\rm H}^1(G,\Aut(X/\C))$.  The stabilizer of an element $f\in Z^1(G, \Aut(X/\C))$ is equal to
$$
\{h\in \Aut(X/\C)\mid (f\circ \sigma_0)\circ h=h\circ (f\circ\sigma_0)\}\, ,
$$
which can be identified with the automorphism group of the real form $Y=X/\langle f\circ \sigma_0\rangle$ corresponding to the real structure $f\circ \sigma_0$ of $X$.  Thus, by counting the cardinalities of the orbits, we find
$$
\#\; Z^1(G, \Aut(X/\C))=\sum_{Y\in \RF(X)}\frac{\#\; \Aut(X/\C)}{\#\;\Aut(Y/\R)}\, .
$$ 
Dividing by $\# \Aut(X/\C)$ delivers the desired formula.  Since $Z^1(G, \Aut(X/\C))$ is identified with a subset of $\Aut(X/\C)$, the inequality follows.

The obstruction for the product $fg$ of two elements $f,g \in Z^1(G,\Aut(X/\C)) \subseteq \Aut(X/\C)$ to lie in $Z^1(G,\Aut(X/\C))$ is the commutator: $fg \in Z^1(G,\Aut(X/\C))$ if and only if $f\, g\, c(f) \, c(g)=f\, g\, f^{-1} g^{-1}$ is trivial.  The equality $Z^1(G,\Aut(X/\C)) = \Aut(X/\C)$ implies that $\Aut(X/\C)$ is abelian and $c$ is the map $f \mapsto f^{-1}$ on $\Aut(X/\C)$.  This shows the inequality for non-abelian $\Aut(X/\C)$.
\end{proof}

\begin{corollary}
Suppose $\Aut(X/\C)$ is a finite group and $Y$ is a real form of\, $X$.
If $\Aut(Y/\R)=\{\id_Y\}$, then $\Aut(X/\C)$ is abelian and $\# \RF(X)=1$.
\end{corollary}

%%%%%%%%%%%%%%%%%%
\section{A bound using a 2-Sylow subgroup}

Because of the condition $f \circ (c \cdot f)=\id_X$ for $Z^1(G,\Aut(X/{\C}))$, we can often give a better bound.  Indeed, the number of real forms of $X$ with finite $\Aut(X/\C)$ is bounded by the cardinality of the first cohomology set of $G$ in some Sylow $2$-subgroup of $\Aut(X/\C)$.

\begin{proposition}\label{pp:sylow2}
Suppose $\Aut(X/\C)$ is finite. Then the following statements hold:
\begin{enumerate}
  \item\label{pp:sylow2-1} There is a Sylow $2$-subgroup $H\subseteq\Aut(X/\C)$ such that $H$ is a $G$-subgroup of $\Aut(X/\C)$. 
  \item\label{pp:sylow2-2} For such $H$, we have $\# \; \RF(X)\leq \#\; {\rm H}^1(G,H)$.
\end{enumerate}
\end{proposition}

\begin{proof}
Let $K:=\Aut(X/\C)\rtimes \langle \sigma_0\rangle$. By the Sylow theorem, there is a Sylow $2$-subgroup $S<K$ such that $\sigma_0\in S$. Let $H:=S\cap \Aut(X/\C)$.  Clearly, $H$ is a Sylow $2$-subgroup of $\Aut(X/\C)$ and $H\cup H\sigma_0=S$. Then, we have $S=H\rtimes \langle \sigma_0\rangle$, which implies that $H$ is a $G$-subgroup of $\Aut(X/\C)$ and it proves~\eqref{pp:sylow2-1}.

We may identify the set of $1$-cocycles of $G$ in $H$ with  a subset of $H$: 
$$
Z^1(G,H)=\{f\in H\mid f\circ (c\cdot f)=\id_X\}.
$$ 
We define a map
\begin{equation}\label{eq1}
\varphi\colon {\rm H}^1(G,H)\longrightarrow \overline{\RS}(X), \quad [f]\longmapsto [f\circ \sigma_0].
\end{equation}
Indeed, $\varphi$ is well defined. If $[f]=[f^\prime]\in {\rm H}^1(G,H)$, then there is an element $h\in H$ such that $h^{-1}f\, (c\cdot h)=f^\prime$, that is, $h^{-1}f\sigma_0 h\sigma_0=f^\prime$.  Thus, we have $h^{-1}f\sigma_0 h=f^\prime\sigma_0$ and $[f\sigma_0]=[f^\prime \sigma_0]\in \overline{\RS}(X)$.  In fact, $\varphi$ factors via the natural map ${\rm H}^1(G,H) \to {\rm H}^1(G,\Aut(X/\C))$.

Next we show that $\varphi$ is surjective. Let $[\sigma]\in \overline{\RS}(X)$.  Since $\sigma\in K\setminus \Aut(X/\C)$ is of order $2$, by the Sylow theorem again, there exists an $\alpha\in K$ such that $\sigma^\prime:=\alpha \sigma\alpha^{-1}\in S$.  Note that $\sigma^\prime \sigma_0\in H$ and $[\sigma^\prime]=\varphi([\sigma^\prime \sigma_0])\in {\Im}(\varphi)$.  So it suffices to show $[\sigma]=[\sigma^\prime]\in \overline{\RS}(X)$.  If $\alpha\in \Aut(X/\C)$, then clearly $[\sigma^\prime]=[\sigma]\in \overline{\RS}(X)$.  On the other hand, if $\alpha\notin \Aut(X/\C)$, then $\alpha=\beta\sigma_0$ for some $\beta\in \Aut(X/\C)$. Then
$$
\sigma^\prime=\beta\sigma_0\sigma\sigma_0 \beta^{-1}=(\beta\sigma_0\sigma)\sigma (\beta\sigma_0\sigma)^{-1},
$$
which implies $[\sigma^\prime]=[\sigma]\in \overline{\RS}(X)$ since $\beta\sigma_0\sigma\in \Aut(X/\C)$. 

By the surjectivity of $\varphi$ and the equality $\# \RF(X)=\# \overline{\RS}(X)$, we have $\#\; \RF(X)\leq \#\; {\rm H}^1(G,H)$. This completes the proof of~\eqref{pp:sylow2-2}.
\end{proof}

In the setting of Proposition~\ref{pp:sylow2}, the Sylow $2$-subgroup $H$ acts on $Z^1(G,H)$ via $f \mapsto h^{-1}f (c\cdot h)$, and the space of orbits is ${\rm H}^1(G,H)$. We then count orbits to get the following result.

\begin{theorem}\label{pp:sylow3}
Suppose $\Aut(X/{\C})$ is finite and $H\subset \Aut(X/\C)$ is a Sylow $2$-subgroup preserved by $G$ as in Pro\-po\-si\-tion~{\rm\ref{pp:sylow2}}. Then we have
$$
\sum_{[Y]\in {\rm H}^1(G,H)}  \frac{1}{\#(\Aut(Y/{\R})\cap H)} = \frac{\# Z^1(G,H)}{\# H} \leq 1\, .
$$
Moreover, if\, $H$ is not abelian, then the inequality is strict. 
\end{theorem}

\begin{proof}
We let $H$ act on $Z^1(G,H)$ via $Z^1(G,H) \times H \to Z^1(G,H)$, $(f,h) \mapsto h^{-1}f(c \cdot h)$.  Then the set of orbits can be identified with ${\rm H}^1(G,H)$.  The stabilizer of $f\in Z^1(G,H)$ is $\{h \in H\mid h^{-1}f \sigma_0h=f\sigma_0\}$ and can be identified with $\Aut(Y/{\R})\cap H$ with $Y$ the real form $X/\langle f\sigma_0\rangle$.  The number of elements in $Z^1(G,H)$ equals the sum of the lengths of the orbits, and the length of an orbit defined by $[Y]\in {\rm H}^1(G,H)$ is $\# H/\# \Aut(Y/{\R})\cap H$.  We thus get $\#Z^1(G,H)=\sum_{[Y]\in {\rm H}^1(G,H)} \# H/\# \Aut(Y/{\R})\cap H$. Dividing by $\# H$ gives the equality. Since elements of $Z^1(G,H)$ can be identified with elements of $H$, the inequality follows.  The strict inequality for non-abelian $H$ follows in exactly the same manner as in the proof of Theorem \ref{thm:formula}: if $f$ and $g$ are in $Z^1(G,H)$ (viewed as a subset of $H$), then $fg\in Z^1(G,H)$ if and only if $fgf^{-1}g^{-1}$ is trivial.  So if $\# H=\# Z^1(G,H)$, then $H$ must be abelian.
\end{proof}

\begin{remark}
The sum in Theorem \ref{pp:sylow3} is over elements $[Y]\in {\rm H}^1(G,H)$, not over $[Y] \in {\rm H}^1(G,\Aut(X/\C))$; but the map $\varphi\colon {\rm H}^1(G,H) \to {\rm H}^1(G,\Aut(X/\C)) \cong \overline{\RS}(X)$ defined in~\eqref{eq1} is surjective. Thus we have an inequality
$$
\sum_{[Y]\in R}\frac{1}{\#(\Aut(Y/{\R})\cap H)} \leq \frac{\# Z^1(G,H)}{\# H}
$$
with $R\subset {\rm H}^1(G,H)$ a subset such that $\varphi_{|R}\colon R\to \overline{\RS}(X)$ is a bijection. If one has information on the cardinality of the fibres of $\varphi$, one can sharpen Proposition \ref{pp:sylow2}\eqref{pp:sylow2-2} and Theorem \ref{pp:sylow3}. A description of the fibres is given in \cite[Section~I.5.4, Corollary~2]{Se02}.
\end{remark}

\begin{corollary}
Let $X$ be a complex quasi-projective variety with finite automorphism group and $Y$ a real form.  If $\#\Aut(Y/\R)$ is odd, then $\# \Aut(X/\C)$ is odd and $\# {\RF}(X)=1$.
\end{corollary}

\begin{proof}
The left-hand side of the equality in Theorem \ref{pp:sylow3} has only one term; hence we see $\# Z^1(G,H)=\# H$ and $\# {\rm H}^1(G,H)=1$. Since the product of two elements of $Z^1(G,H)$ lies in $Z^1(G,H)$, we see as at the end of the proof of Theorem \ref{thm:formula} that $H$ is abelian and $c(h)=h^{-1}$ for all $h\in H$.  But each element of $H$ of order $2$ defines an element of $Z^1(G,H)={\rm H}^1(G,H)$; hence $H$ is trivial.
\end{proof}

\begin{corollary}
If\, $Z$ is a complex quasi-projective variety with finite $\Aut(Z/\C)$ of odd order, then $Z$ has at most one real form, up to an isomorphism over $\Spec \R$.
\end{corollary}

%%%%%%%%%%%%%%%%%%%%%%%%%%%%%
\section{An application: Real forms of plane curves}
We prove a bound on the number of real forms of a plane curve. We start with remarks on the cohomology of $G$-sets.

Let $H$ be a finite group, with automorphism group $\Aut(H)$, and let $\phi\colon G\rightarrow \Aut(H)$ be a group homomorphism. Then we may view $H$ as a $G$-group via $\phi$, and the first cohomology set ${\rm H}^1(G, H)$ is defined as in \cite{Se02}.  We sometimes denote ${\rm H}^1(G,H)$ by ${\rm H}^1(G, H_\phi)$ to emphasize that the underlying action is $\phi$.

\begin{definition}
Let $H$ be a finite group. We use $m(H)$ to denote the maximal cardinality of ${\rm H}^1(G, H_\phi)$ when $\phi$ runs over all possible group homomorphisms $\phi\colon G\rightarrow \Aut(H)$.  We denote the minimal number of generators of a finite group $H$ by $l(H)$.
\end{definition}

Since $G\cong \Z/2\Z$, the set of group homorphisms from $G$ to $\Aut(H)$ corresponds bijectively to the set of elements in $\Aut(H)$ of order at most $2$.  If two homomorphisms $\phi_i\colon G\rightarrow \Aut(H)$ $(i=1, 2)$ have the property that $\phi_1(c)$ and $\phi_2(c)$ are conjugate in $\Aut(H)$, then $\#\; {\rm H}^1(G, H_{\phi_1})=\#\; {\rm H}^1(G, H_{\phi_2}).$

\begin{lemma}\label{lem:ab}
Let $H$ be a finite abelian group. Then $m(H)=2^{l(H_2)}$, where $H_2$ is a Sylow $2$-subgroup of\, $H$.
\end{lemma}

\begin{proof}
For the trivial action of $G$ on $H$, the cardinality of ${\rm H}^1(G, H_{\triv})$ is equal to the number of conjugacy classes of elements of order at most $2$ in $H$.  The group $H_2$ is isomorphic to $\Z/d_1\Z\times \cdots\times \Z/d_r\Z$, where $d_i$ divides $d_{i+1}$ for each $1 \leq i\leq r-1$, and the elements of order at most $2$ are obtained by taking the elements of order at most $2$ in each $\Z/d_i\Z$.  This implies $\# \;{\rm H}^1(G, H_{\triv})=2^{l(H_2)}$.  Thus we have the lower bound $m(H)\ge 2^{l(H_2)}$.

Let $\phi\colon G\rightarrow \Aut(H)$ be a group homomorphism. Since $H$ is abelian, we write $H$ additively, and then
$$
Z^1(G, H)=\{f\in H\mid \; 0=c\cdot f+f=(c+1)\cdot f\}
$$
is a subgroup of $H$, where $c\cdot x=\phi(c)(x)$ for any $x\in H$.  Moreover, ${\rm H}^1(G,H)=Z^1(G,H)/K$, where $K=\{(c-1)h\mid h\in H \}$. The quotient group $Z^1(G,H)/K$ is of exponent at most $2$; in fact, if $f\in Z^1(G,H)$, then $2f=(c-1)(-f)\in K$.  We denote by $(Z^1(G,H)/K)_2$ (resp.\ $Z^1(G,H)_2$) a Sylow $2$-subgroup of $Z^1(G,H)/K$ (resp.\ $Z^1(G,H)$). Note that $Z^1(G,H)_2$ is mapped surjectively to $(Z^1(G,H)/K)_2$ under the natural projection $Z^1(G,H)\rightarrow Z^1(G,H)/K$.  Then we have
$$
l(Z^1(G,H)/K)=l((Z^1(G,H)/K)_2)\le l(Z^1(G,H)_2)\le l(H_2)\, .
$$
Thus, $\# \; {\rm H}^1(G,H)=\# \; Z^1(G,H)/K=2^{l(Z^1(G,H)/K)} \le 2^{l(H_2)}$; 
this gives the upper bound $m(H)\le 2^{l(H_2)}$. We conclude $m(H)= 2^{l(H_2)}$.
\end{proof}

Recall that a subgroup $K$ of a group $H$ is called a {\it characteristic subgroup} if $\varphi(K)=K$ for every $\varphi \in \Aut(H)$. Clearly, a characteristic subgroup is a normal subgroup.

\begin{lemma}\label{lem:charsub}
Let $K$ be a characteristic subgroup of a finite group $H$. Then
$$
m(H)\le m(K)\,  m(H/K).
$$
\end{lemma}

\begin{proof}
Let $\phi\colon G\rightarrow \Aut(H)$ be a group homomorphism.  Then $H$ is a $G$-group and $K$ is a $G$-subgroup of $H$ since $K$ is a characteristic subgroup.  The resulting exact sequence of cohomology sets
$$
{\rm H}^1(G,K)\longrightarrow {\rm H}^1(G,H)\longrightarrow {\rm H}^1(G, H/K)
$$
implies by \cite[Section I.5.5, Corollary 2]{Se02} that
$$
\#\; {\rm H}^1(G,H)\le m(K) \#\; {\rm H}^1(G, H/K)\le m(K)\, m(H/K)\, .  
$$
From this, we conclude that $m(H)\le m(K)\, m(H/K).$
\end{proof}

Using this lemma, we can compute $m(D_{2n})$ for the dihedral group $D_{2n}$ of order $2n$.

\begin{lemma}\label{lem:dih} Let $n\geq 3$. Then 
  \[
  m(D_{2n})=\begin{cases} 2, & n\equiv 1 \; (\bmod \; 2), \\
4, & n\equiv 0 \; (\bmod \; 2). \\ 
\end{cases} \]
\end{lemma}

\begin{proof}
If $n$ is odd (resp.\ even), then the number of conjugacy classes of elements of $H=D_{2n}$ of order at most~$2$ is two (resp.\ four), which implies $m(H)\ge 2$ (resp.\ $4$).  Let $K$ denote the unique cyclic subgroup of $H$ of order $n$.  Since $K$ is a characteristic subgroup, Lemma \ref{lem:charsub} implies that $m(H)\le m(K)\, m(H/K)=2m(K)$. Then $m(K)=1$ (resp.\ $2$) by Lemma \ref{lem:ab}.  From this we have $m(H)\le 2$ (resp.\ $4$). Thus, we have $m(H)=2$ (resp.\ $4$) if $n$ is odd (resp.\ even).
\end{proof}

\begin{example}\label{ex:Q8}
Consider the quaternion group $H=Q_8=\{\pm 1,\pm i,\pm j,\pm k\}$.  It is known that $\Aut(Q_8)$ is isomorphic to the symmetric group $\mathfrak{S}_4$, which has exactly three conjugacy classes of elements of order at most $2$.  Their representatives give three homomorphisms $\phi_t\colon G\rightarrow \Aut(Q_8)$ ($t=1,2,3$) with $\phi_1(c)=\id_{Q_8}$ and $\phi_2(c)$, $\phi_3(c)$ given by
$$
\phi_2(c)\colon   i\longmapsto i,\ j\longmapsto -j,\ k\longmapsto -k, \quad
\phi_3(c)\colon   i\longmapsto j,\ j\longmapsto i,\ k\longmapsto -k.
$$
A computation shows $\#\; {\rm H}^1(G, {H}_{\phi_t})=2,3,1$ for $t=1,2,3$, respectively.  In particular, $m(Q_8)=3>2$, which is the number of conjugacy classes of elements of order at most $2$ in $Q_8$.
\end{example}

Badr and Bars proved that $(\Z/2\Z)^2$ cannot act faithfully on any smooth plane curve of degree $d=5$. By a similar argument, the same result holds for any odd $d\ge 5$.

\begin{lemma}[{cf. \cite[Lemma 10]{BB16}}]\label{lem:odd}
The group $(\Z/2\Z)^2$ is not isomorphic to a subgroup of the automorphism group of any smooth plane curve of odd degree $d\ge 5$.
\end{lemma}

\begin{lemma}\label{lem:a-ii}
Let $H\subset \GL(2,\C)$ be a finite subgroup of order $2^s$ $(s\ge 3)$. Let $\pi\colon \GL(2,\C)\rightarrow \PGL(2,\C)$ be the natural projection map. If\, $\pi(H)$ is isomorphic to a dihedral group of order $2^{s-1}$, then $m(H)\le 4$.
\end{lemma}

\begin{proof}
A finite group of order $8$ is one of the following groups: $\Z/8\Z$, $\Z/4\Z\times \Z/2\Z$, $D_8$, $Q_8$, $(\Z/2\Z)^3$. Since $(\Z/2\Z)^3$ is not a subgroup of $\GL(2,\C)$, we have $H\cong$ $\Z/4\Z\times \Z/2\Z$, $D_8$, $Q_8$ if $s=3$. Then $m(H)\le 4$ by Lemmas \ref{lem:ab} and \ref{lem:dih} and Example \ref{ex:Q8}.  From now on, we assume $s\ge 4$. By \cite[Theorem 4(4)(a)(ii)]{NPT08}, $H$ is conjugate (in $\GL(2,\C)$) to one of the following three groups:
$$
\left\langle i\begin{pmatrix}\xi_{2^{s-1}}&0\\0&\xi_{2^{s-1}}^{-1}\end{pmatrix}, 
\begin{pmatrix}0&i\\i&0\end{pmatrix}\right\rangle,\quad \left\langle \begin{pmatrix}\xi_{2^{s-1}}&0\\0&\xi_{2^{s-1}}^{-1}\end{pmatrix}, 
\begin{pmatrix}0&1\\1&0\end{pmatrix}\right\rangle, \quad \left\langle \begin{pmatrix}\xi_{2^{s-1}}&0\\0&\xi_{2^{s-1}}^{-1}\end{pmatrix}, 
\begin{pmatrix}0&i\\i&0\end{pmatrix}\right\rangle,
  $$
where $\xi_n$ indicates a primitive $\supth{n}$ root of unity.  From this, we see that $H$ has a unique cyclic subgroup of order $2^{s-1}$, say $N$.  Then $N$ is a characteristic subgroup of $H$ with $H/N\cong \Z/2\Z$. Thus, by Lemma \ref{lem:charsub} we find $m(H)\le 4$.
\end{proof}

\begin{theorem}\label{thm:planecurve}
Let $X$ be a smooth plane curve of degree $d\ge 4$. Then 
$$
\# {\RF}(X) \leq \begin{cases} 2, & d \equiv 1 \, (\bmod \; 2),  \\ 
4, & d \equiv 2 \; (\bmod \;4), \\
8, & d \equiv 0 \; (\bmod  \;4). \\ \end{cases}
$$
\end{theorem}

\begin{proof}
Let $H=\Aut(X/\C)$ and let $H_2$ be a Sylow $2$-subgroup of $H$. We use the results of Harui \cite{Ha19} on automorphism groups of plane curves.

\begin{enumerate}[wide,itemindent=16pt]
\item\label{p4.8-1} Suppose $d$ is odd. Then it suffices to show $m(H)\le 2$. By \cite[Theorem 2.3]{Ha19}, we only need to consider five cases: (a-i), (a-ii), (b-i), (b-ii), (c) as listed there.

(a-i)~ The group $H$ is cyclic. Then $m(H)\le 2$ by Lemma \ref{lem:ab}.

(a-ii)~ The group $H$ fits in the exact sequence 
$$
1\longrightarrow N \longrightarrow H\longrightarrow K\longrightarrow 1,
$$
where $N$ is a cyclic group of order dividing $d$ and $K$ is isomorphic to one of the following groups: $\Z/n\Z$, $D_{2n}$, $\mathfrak{A}_4$, $\mathfrak{S}_4$, $\mathfrak{A}_5$. Thus, $H$ and $K$ have isomorphic Sylow $2$-subgroups. Lemma \ref{lem:odd} does not allow $(\Z/2\Z)^2$ as a subgroup of $H$.  So $K$ is either cyclic or $D_{2n}$ with odd $n$, and thus $H_2$ is a cyclic group.  By Lemma \ref{lem:ab}, we have $m(H)\le m(H_2)\le 2$.

(b-i)~ The group $H$ is isomorphic to a subgroup of $\Aut(F_d)\cong (\Z/d\Z)^2\rtimes \mathfrak{S}_3$, where $F_d$ is the Fermat curve defined by $x^d+y^d+z^d=0$. Since $d$ is odd, $H_2$ is of order at most $2$. Then $m(H)\le m(H_2)\le 2$.

(b-ii)~ The group $H$ is isomorphic to a subgroup of $\Aut(C_d)$, where $C_d$ is the Klein curve defined by $xy^{d-1}+yz^{d-1}+zx^{d-1}=0$. Since $\Aut(C_d)$ is of odd order $3(d^2-3d+3)$ (see \cite[Proposition 3.5]{Ha19}), the group $H$ is of odd order and $m(H)=1$.

(c)~ The group $H$ is isomorphic to one of the following groups: $\mathfrak{A}_5$, $\PSL(2,\F_7)$, $\mathfrak{A}_6$, $H_{216}$, $(\Z/3\Z)^2\rtimes \Z/4\Z$, $(\Z/3\Z)^2\rtimes Q_8$.  Here $H_{216}$ is the Hessian group of order $216$, and the last two groups are subgroups of $H_{216}$.

By Lemmas \ref{lem:ab} and \ref{lem:charsub}, we infer that $m((\Z/3\Z)^2\rtimes \Z/4\Z)=2$. For $H=\mathfrak{A}_5$, $\PSL(2,\F_7)$, $\mathfrak{A}_6, H_{216}$, $(\Z/3\Z)^2\rtimes Q_8$, the automorphism group $\Aut(H)$ is isomorphic to $\mathfrak{S}_5$, $\PSL(2,\F_7)\rtimes \Z/2\Z$, $\mathfrak{S}_6\rtimes \Z/2\Z$, $H_{216}\rtimes \Z/2\Z$, $H_{216}\rtimes \Z/2\Z$, respectively. Moreover, the number of conjugacy classes of elements of order at most $2$ in $\Aut(H)$ is equal to $3$, $3$, $4$, $3$, $3$, respectively. Then we conclude that $m(H)=2$ by a computation as the one in Example \ref{ex:Q8}. We give the details for the case $H=\mathfrak{A}_5$. Note that we have the isomorphism $\psi\colon \mathfrak{S}_5\rightarrow \Aut(\mathfrak{A}_5)$ given by $\psi(\alpha)(\beta)=\alpha \beta \alpha^{-1}$ for any $\alpha\in \mathfrak{S}_5$, $\beta\in
\mathfrak{A}_5$. Then $\Aut(\mathfrak{A}_5)$ has three conjugacy classes of elements of order at most $2$ which are represented by $\psi(\alpha_1)$, $\psi(\alpha_2)$, $\psi(\alpha_3)$, where $\alpha_1=(1)$, $\alpha_2=(12)(34)$, $\alpha_3=(12)$. These representatives give three homomorphisms $\phi_t\colon G\rightarrow \Aut(\mathfrak{A}_5)$ with $\phi_t(c)=\psi(\alpha_t)$ ($t=1,2,3$). Clearly, we have $\#{\rm H}^1(G, H_{\phi_1})=\#{\rm H}^1(G, H_{\triv})=2$. Since the automorphism $\psi(\alpha_2)$ is an inner automorphism given by the conjugation action of the element $\alpha_2$ in $H=\mathfrak{A}_5$, by \cite[Proposition I.35 bis]{Se02}, we have $\#{\rm H}^1(G, H_{\phi_2})=\#{\rm H}^1(G, H_{\triv})=2$. For $\#{\rm H}^1(G, H_{\phi_3})$, we observe that $H_2:=\{(1), (12)(34), (13)(24), (14)(23)\}$ is a Sylow $2$-subgroup of $H$ preserved by $G$, and a computation shows that $Z^1(G, H_2)=\{(1), (12)(34)\}$, which implies $\#{\rm H}^1(G, H_{\phi_3})\le \#{\rm H}^1(G, H_2)\le
\#Z^1(G, H_2)=2$. Therefore, we conclude that $m(\mathfrak{A}_5)=2$. The cases $H=\PSL(2,\F_7)$, $\mathfrak{A}_6, H_{216}$, $(\Z/3\Z)^2\rtimes Q_8$ can be treated in a similar~way.

\item\label{p4.8-2} Suppose $d=4k+2$ with $k\ge 1$. We apply \cite[Theorem 2.3]{Ha19} again.  As in~\eqref{p4.8-1}, for cases (a-i), (b-ii) and (c), we have $m(H)\le 2$.  Then we only need to consider cases (a-ii) and~(b-i).

(a-ii)~  We consider the subgroup of $\PGL(3,\C)$ consisting of elements represented by matrices of the form $\begin{psmallmatrix}A&0\\0&1\end{psmallmatrix}$ with $A\in\GL(2,\C)$.  Under identification of this group with $\GL(2,\C)$, the natural projection $\pi\colon \GL(2,\C)\rightarrow \PGL(2,\C)$ coincides with the map $\rho \colon \PBD(2,1)\rightarrow \PGL(2,\C)$ in \cite[Theorem 2.3]{Ha19}.  Then $H$ can be viewed as a subgroup of $\GL(2,\C)$, and it fits in the exact sequence
  $$
  1\longrightarrow N \longrightarrow H\xrightarrow{^{\pi|H}}  K\longrightarrow 1,
  $$
where $K\subset \PGL(2,\C)$ is isomorphic to one of the following groups: $\Z/n\Z$, $D_{2n}$, $\mathfrak{A}_4$, $\mathfrak{S}_4$, $\mathfrak{A}_5$. Since the order $|N|$ of $N$ divides $d$ (see \cite[Lemma 3.8]{Ha19}), we have $4\nmid |N|$. If $2\nmid |N|$, then $H_2\cong K_2$ and we have $m(H)\le m(H_2)=m(K_2)\le 4$.  Thus, we only need to consider the case where $2$ divides $|N|$. Then $H_2$ fits in the exact sequence
$$
1\longrightarrow \Z/2\Z \longrightarrow H_2\xrightarrow{^{\pi|H_2}}  K_2\longrightarrow 1. 
$$ Clearly, $K_2$ is either cyclic or a dihedral group. If $K_2$ is cyclic, then $H_2$ is abelian and $m(H_2)\le 4$. So we may assume $K_2\cong D_{2^s}$ for some $s\ge 2$. Then by Lemma \ref{lem:a-ii}, we have $4\ge m(H_2)\ge m(H)$.

(b-i)~ The group $H$ is isomorphic to a subgroup of $(\Z/d\Z)^2\rtimes \mathfrak{S}_3$.  Since a Sylow $2$-subgroup of $(\Z/d\Z)^2\rtimes \mathfrak{S}_3$ is isomorphic to $D_8$, the group $H_2$ is a subgroup of $D_8$, which implies that $m(H)\le m(H_2)\le 4$.

This completes the proof for $d \equiv 2 \; (\bmod \; 4)$.

\item\label{p4.8-3} Suppose $d=4k$ with $k\ge 1$. We claim that $m(M)\le 8$ for any finite subgroup $M$ in $\GL(2,\C)$. If $M$ is abelian, then we may assume that $M$ consists of diagonal matrices, which implies that $M$ is generated by two elements and $m(M)\le 4$ by Lemma \ref{lem:ab}. If $M$ is not abelian, then the center $Z(M)$ of $M$ is a cyclic group consisting of scalar matrices and the quotient group $M/Z(M)$ is isomorphic to a finite subgroup in $\PGL(2,\C)$ (\textit{i.e.} $M/Z(M)$ is isomorphic to one of the following groups: $\Z/n\Z$, $D_{2n}$, $\mathfrak{A}_4$, $\mathfrak{S}_4$, $\mathfrak{A}_5$). Thus, by Lemma \ref{lem:charsub}, we infer that $m(M)\le m(Z(M)) m(M/Z(M))\le 2\cdot 4=8$. This completes the proof of the claim.

Next we apply \cite[Theorem 2.3]{Ha19} as in \eqref{p4.8-1} and \eqref{p4.8-2}. As in~\eqref{p4.8-1}, for cases (a-i) and~(c), we have $m(H)\le 2$. Then we only need to consider cases~(a-ii), (b-i) and~(b-ii).

(a-ii)~ As in~\eqref{p4.8-2}, $H$ can be viewed as a subgroup of $\GL(2,\C)$. Then $m(H)\le 8$ by the previous claim.

(b-i)~ As in~\eqref{p4.8-1}, $H$ is isomorphic to a subgroup of $\Aut(F_d)\cong (\Z/d\Z)^2\rtimes \mathfrak{S}_3$. Consider the finite subgroup $M_1\subset \PGL(3, \C)$ generated by the following matrices: $\diag(\xi_{d},1,1)$, $\diag(1,\xi_d,1)$, $\begin{psmallmatrix}0& 1&0\\1&0&0\\0&0&1\end{psmallmatrix}$. Here $\xi_d$ is a primitive $\supth{d}$ root of unity. Clearly, $M_1 \cong (\Z/d\Z)^2\rtimes \Z/2\Z$ contains a Sylow $2$-subgroup of $\Aut(F_d)$, and $M_1$ can be viewed as a finite subgroup in $\GL(2,\C)$. Thus, a Sylow $2$-subgroup of $H$ is isomorphic to a subgroup in $\GL(2,\C)$. Then by Proposition \ref{pp:sylow2} and the previous claim, we have $m(H)\le 8$.

(b-ii)~ As in~\eqref{p4.8-1}, if $d\ge 8$, then the order $|H|$ is odd, which implies $m(H)=1$. If $d=4$, then $H$ is isomorphic to a subgroup of $\PSL(2, \F_7)$ (see \cite[Theorem 2.5]{Ha19}), which implies $m(H)\le 4$ since a Sylow $2$-subgroup of $\PSL(2, \F_7)$ is isomorphic to $D_8$.

  This completes the proof for $d \equiv 0 \; (\bmod \; 4)$.\qedhere
  \end{enumerate}
\end{proof}

\begin{remark}
Recently Sasaki \cite{Sa23} computed the number of real forms of all Fermat hypersurfaces of degree at least $3$. In particular, for Fermat curves $F_d$ he showed that $\#\; \RF(F_d)=2$ (resp.\ $3$) if $d$ is odd (resp.\ even).  Thus, our bound in Theorem \ref{thm:planecurve} is optimal for odd $d$.  For even $d\geq 6$, we can prove by a computation similar to the one in \cite{Sa23} that the curve $X_d$ defined by $x^d+y^d+z^d+y^2z^{d-2}=0$ has four real forms.  In fact, $\Aut(X_d/\C)$ is isomorphic to $\Z/d\Z \times \Z/2\Z$, and it is generated by $(x: y: z) \mapsto (\xi_d x: y: z)$ and $(x: y: z) \mapsto (x: -y: z)$ with $\xi_d$ a primitive $\supth{d}$ root of unity. This implies that our bound in Theorem~\ref{thm:planecurve} for $d\equiv 2 \, (\bmod\, 4)$ is optimal.  For $d\equiv 0 \, (\bmod\, 4)$, we speculate that the optimal bound is $6$ instead of $8$.
\end{remark}

%%%%%%%%%%%%%%%%%%%%%
% References
%%%%%%%%%%%%%%%%%%%%%


\begin{thebibliography}{NvdPT08+++} 
      
\bibitem[BB16]{BB16} E.\ Badr and F.\ Bars, 
\emph{Automorphism groups of nonsingular plane curves of degree 5}, Comm.\ Algebra \textbf{44} (2016), no.~10, 4327--4340, \doi{10.1080/00927872.2015.1087547}.

\bibitem[vdGV92]{GV92} G.\ van der Geer and M.\ van der Vlugt, \emph{Supersingular curves of genus 2 over finite fields of characteristic 2}, 
Math.\ Nachr.\ \textbf{159} (1992), 73—-81, \doi{10.1002/mana.19921590106}.

\bibitem[GI98]{GI98} G.\ Gromadzki and M.\ Izquierdo, \emph{Real forms of a Riemann surface of even genus}, Proc.\ Amer.\ Math.\ Soc.~\textbf{126} (1998), no.~12, 3475--3479, \doi{10.1090/S0002-9939-98-04735-2}.

\bibitem[Har19]{Ha19} T.\ Harui, \emph{Automorphism groups of smooth plane curves}, Kodai Math.\ J.~\textbf{42} (2019), no.~2, 308--331, \doi{10.2996/kmj/1562032832}.

\bibitem[Nat78]{Na78} S.\,M.\ Natanzon, \emph{On the order of a finite group of homeomorphisms of a surface into itself, and the real number of real forms of a complex algebraic curve}, Dokl.\ Akad.\ Nauk SSSR \textbf{242} (1978), no.~4, 765–-768 (Russian); Sov.\ Math., Dokl.\ \textbf{19} (1978), no.~5, 1195–-1199 (English).

\bibitem[NvdPT08]{NPT08} K.\,A.\ Nguyen, M.\ van der Put, and J.\ Top,  
\emph{Algebraic subgroups of $\GL_2(\C)$}, Indag.\ Math.\ (N.S.) \textbf{19} (2008), no.~2, 287--297, \doi{10.1016/S0019-3577(08)80004-3}.

\bibitem[Sas23]{Sa23} Y.\ Sasaki, \emph{On real forms of Fermat hypersurfaces}, prepint \arXiv{2312.10691} (2023).

\bibitem[Ser02]{Se02}
J.-P.\ Serre, \emph{Galois cohomology} (translated from the French by P.~Ion and revised by the author), Springer Monogr.\ Math., Springer-Verlag, Berlin, 2002. 
\end{thebibliography}
\end{document}